\documentclass{article}
\usepackage{enumerate}
\usepackage{amsmath, amssymb}
\usepackage{amsthm}
\usepackage{graphicx}

\renewcommand{\Re}{\mathrm{Re}}
\renewcommand{\Im}{\mathrm{Im}}
\renewcommand{\phi}{\varphi}

\newcommand{\quotes}[1]{``#1''}
\newcommand{\st}{ \ \mathrm{s.t.} \ }

\theoremstyle{definition}
\renewcommand{\proofname}{Proof}
\newcommand{\defname}{Definition}
\newtheorem{Def}{\defname}
\newcommand{\propname}{Proposition}
\newtheorem{prop}{\propname}
\newcommand{\theoname}{Theorem}
\newtheorem{thm}{\theoname}
\newcommand{\corname}{Corollary}
\newtheorem{cor}{\corname}

\DeclareMathOperator{\ord}{ord}

\newcommand{\set}[1]{\{ #1 \}}
\newcommand{\del}{\partial}
\newcommand{\R}{\mathcal{R}}

\newcommand{\GL}{\mathrm{GL}}
\newcommand{\RR}{\mathbf{R}}
\newcommand{\CC}{\mathbf{C}}
\newcommand{\zero}{\mathbf{0}}
\newcommand{\m}{\mathfrak{m}}
\newcommand{\E}{\mathcal{E}}
\newcommand{\Harm}{\mathcal{H}}

\title{Singularities of functions with harmonic leading terms of two variables}
\author{Yuki Yasuda}
\begin{document}
\maketitle

\begin{abstract}
We study the classification problem of singularities of function-germs with harmonic leading terms
  of two variables under the right-equivalence.
We study the classification in the cases that the order of function-germs is at most 7.
Moreover, we observe that the multiple actions of Laplacian appear for the classifications
  of such class of function-germs (\theoname{s}\  \ref{main-5} - \ref{main-7}).
\end{abstract}

\section{Introduction}

In this paper, we study the classification of singularities of function-germs of two variables with
  a non-zero harmonic leading term under the right-equivalence.

We study the classification in the cases that the order of function-germs is at most 7.
Moreover, we observe that the multiple actions of Laplacian appear for the classifications
  of such class of function-germs (see \theoname{s}\  \ref{main-5} - \ref{main-7}).

For example, we prove that singularities of functions with harmonic leading
  terms of order 5 are completely classified by the action of $\Delta^3$ to the term of
  degree 6.

Recall that a function $h$ on $\RR^n$ is called a harmonic function if it satisfies the Laplace
equation $\Delta h = 0$, where $\Delta = \sum_{j=1}^n \frac{\del^2}{\del x_j^2}$ denotes the
  Laplacian.
In this paper, we study singularities of functions of two variables $x$ and $y$.
Therefore we put $\Delta = \frac{\del^2}{\del x^2} + \frac{\del^2}{\del y^2}$.

Recall that any harmonic function-germ on $(\RR^2, \zero)$ is a real part of
a halomorphic function on $(\CC, 0)$.

Thoughout this paper, we use the notation for the harmonic polynomials of special type:
\[
  f_k = \Re (x + iy)^k, \quad
  g_k = \Im (x + iy)^k = \Re (-i(x+iy)^k),
\]
where $i=\sqrt{-1}$ and $\Re$ (resp. $\Im$) means the real (resp. imaginary) part.

Let $\Harm_k(\RR^2)$ denote the vector space of all homogeneous harmonic polynomial-germs
  from $(\RR^2,\zero)$ to $(\RR, \zero)$ of degree $k$.
  It is known that $\Harm_k(\RR^2)$ is spanned by $f_k$ and $g_k$.
  In particular, $\dim_\RR \Harm_k(\RR^2) = 2$.

By the following propositon, we obtain the right-classification of
  harmonic polynomial space $\Harm_k(\RR^2)\setminus \set{0}$.
This proposition provides the motivation to classify function-germs $h \colon (\RR^2,\zero) \to (\RR,\zero)$ with
  harmonic leading terms with real variables under right equivalent.

\begin{prop} \label{classification-harmonic}
Let $h \colon (\RR^2,\zero) \to (\RR,\zero)$ be a harmonic function-germ, and $k$ be a natural
  number.
Suppose that the order of $h$ at the origin is equal to $k$.
Then $h$ is right-equivalent to $f_k$ by a conformal diffeomorphism-germ $(\RR^2,\zero) \to
  (\RR^2, \zero)$.

In particular, any element $h \in \Harm_k(\RR^2)\setminus\set{0}$ is right equivalent to
  $f_k \in \Harm_k(\RR^2)\setminus \set{0}$ by a linear conformal diffeomorphism-germ.
\end{prop}

\begin{proof}

It is known that for every holomorphic function germ $h_i \colon (\CC,0) \to (\CC, 0)
  (i=1, 2)$, they are holomorphically right-equivalent if and only if $\ord(h_1) = \ord(h_2)$ (see \cite{Bruce-1}).
By considering real part of $h_1$ and $h_2$, we obtain above statement.
\end{proof}

Note that for any harmonic function-germ $h$ and for any linear conformal transformation $\phi$,
  $h \circ \phi$ is a harmonic function-germ.

Let $h \colon (\RR^2,\zero) \to (\RR,0)$ be a $C^\infty$ function germ such
  that
\[
  h = h_k + R_{k+1},
\]
where $h_k$ is a non-zero homogeneous harmonic polynomial of degree $k$ and
  $\ord (R_{k + 1}) \ge k + 1$.

If $k=1$, since $h_1$ is a non-zero linear function, the classification of above function-germs
  is obtained by implicit function theorem i.e. $h$ is right-equivalent to $x$.

If $k=2$, since $h_2$ is a non-degenerate quadratic form, the classification of
  this class is obtained by Morse lemma i.e. $h$ is right-equivalent to $2xy$.

If $k=3$, the singularity of $h_3$ is called $D_4^{-}$. It is proved that $h$ is right
  equivalent to $x^2y - y^3$.

If $k=4$, the singularity of $h$ is of $X_{1, 0}$ type (see \cite{Arnold-1}). There exists a
  real number $a_0$ such that $a_0^2 \neq 4$ and $h$ is right equivalent to $x^4 - a_0x^2y^2 + y^4$
  (in fact, $a_0 = -6$) in our case.

If degree of leading terms is equal to 5, the singularity is called of  $N_{16}$ type
  (see \cite{Arnold-1}). There exists a real number $a, b, c$ such that $ab \neq 9$,
  $4(a^3 + b^3) + 27 - a^2b^2 - 18ab \neq 0$ and $h$ is right equivalent to
  $x^4y + ax^3y^2 + bx^2y^3 + xy^4 + cx^3y^3$.
Note that in \cite{Arnold-1}, several invariants such as multiplicity and moduli number are studied for
  $N_{16}$.

The following result gives the complete solution to our classification problem
  in the case $k=5$:
\begin{thm} \label{main-5}
\begin{enumerate}[(1)]
\item Any function-germ $h=h_5 + R_6\colon (\RR^2,\zero) \to (\RR,0)$ with non-\nobreak zero harmonic
  leading term $h_5$ of degree 5 with $\ord(R_6) \ge 6$, is right equivalent to
  $f_5 + \tilde{R}_6$ for some function-germ $\tilde{R}_6$ with $\ord({\tilde{R}_6})\ge 6$ by a
  linear conformal diffeomorphism-germ.
    \label{main-5-change_coord}
\item For any homogeneous polynomial $\rho_6$ of degree 6 and function-germ
  $R_7$ with $\ord(R_7) \ge 7$, the function-germ $f_5 + \rho_6 + R_7$ is right equivalent to
  $f_5 + \frac{c}{6!}x^6$ where $c=\Delta^3\rho_6 \in \RR$.
    \label{main-5-6}
\item For any $c, \tilde{c} \in \RR$, $f_5 + cx^6$ is right equivalent to $f_5 + \tilde{c}x^6$
  if and only if $c=\tilde{c}$.
    \label{main-5-uniqueness}
\end{enumerate}
\end{thm}

Other theorems (\theoname \ \ref{main-6}, \theoname \ \ref{main-7}) give pre-normal forms of
  function-germs in our classes for $k=6, 7$.
  These theorems give us hidden relations between singularities of functions of our type
  and actions of Laplacian of several times.

To state \theoname{s} \ \ref{main-6}, \ref{main-7} we recall basic terms of singularity theory.


We define $k$-jet of $h$ at the origin $j^kh(\zero)$ by the Taylor series of $h\colon (\RR^2,\zero)
  \to (\RR, 0)$ considered up to the degree $k$.

We say that two function germs $h_i (i=1, 2)$ are $k$-jet equivalent (written $h_1 \cong_{j^k} h_2$)
  if $j^k h_1(\zero) = j^k h_2(\zero)$.

We say that $h_i\colon (\RR^m\zero) \to (\RR^n, \zero)(i=1, 2)$ are $\R_k$ equivalent (written $h_1
  \cong_{\R_k} h_2$) if there exists a diffeomorphism-germ $\tau \colon (\RR^m,\zero) \to
  (\RR^m,\zero)$ and $\tilde{h}_2\colon (\RR^m, \zero) \to (\RR^n, \zero)$
  such that $h_2$ and $\tilde{h}_2$ are $k$-jet equivalent and $h_1$ is right-equivalent to
  $\tilde{h}_2$.

\begin{thm} \label{main-6}
\begin{enumerate}[(1)]
\item Any function germ $h=h_6 + R_7 \colon (\RR^2,\zero) \to (\RR,0)$ with non-zero harmonic
  leading term $h_6$ of degree 6 with $\ord(R_7) \ge 7$ is right equivalent to $g_6 + \tilde{R}_7$
  for some function-germ $\tilde{R}_7$ with $\ord(\tilde{R}_7) \ge 7$ by a conformal
  diffeomorphism-germ.
    \label{main-6_change_leading}
\item For any homogeneous polynomial $\rho_7$ of degree 7 and function-germ $R_8$ with
  $\ord(R_8) \ge 8$, $g_6 + \rho_7 + R_8$ is $\R_7$-equivalent to $g_6 + \frac{c_1}{7!}x^7 +
    \frac{c_2}{6!1!}x^6y$ where
          $c_1 = \frac{\del}{\del x}\Delta^3 \rho_7 \in \RR$,
          $c_2 = \frac{\del}{\del y}\Delta^3 \rho_7 \in \RR$.
      \label{main-6_7}
\item
  For any homogeneous polynomial $\rho_7, \rho_8$ of degree 7 and 8 and function-germ $R_8$
    with $\ord(R_8) \ge 8$,
    $g_6 + \rho_7 + \rho_8$ is $\R_8$- equivalent to $g_6 + \rho_7 + \frac{c}{8!}x^8$ where
    $c=\Delta^4 \rho_8 \in \RR$.
    \label{main-6_8}
\end{enumerate}
\end{thm}

\begin{thm} \label{main-7}
\begin{enumerate}[(1)]
\item Any function-germ $h=h_7 + R_8$ with non-zero harmonic leading term $h_7$ of degree 7 and
  $\ord(R_8) \ge 8$ is right equivalent to $g_7 + \tilde{R}_8$ for some function-germ $\tilde{R}_8$
  with $\ord (\tilde{R}_8)\ge 8$ by a linear conformal diffeomorphism-germ.
  \label{main-7_change_leading}
\item For any homogeneous polynomial $\rho_8$ and fucntion-germ $R_9$ with $\ord (R_9)\ge 9$
  $g_7 + \rho_8$ is $\R_8$-equivalent to $g_7 + \frac{c_1}{8!}x^8 + \frac{c_2}{7!1!}x^7y
  + \frac{c_3}{6!2!}x^6y^2$,
    where homogeneous
    \begin{align*}
    c_1 &= (\frac{\del^2}{\del x^2}-3\frac{\del^2}{\del y^2})\Delta^3 \rho_8 \in \RR, &
    c_2 &= \frac{\del^2}{\del x \del y} \Delta^3 \rho_8 \in \RR, &
    c_3 &= \frac{\del^2}{\del y^2} \Delta^3 \rho_8 \in \RR.
    \end{align*} \label{main-7_8}
\item For any homogeneous polynomial $\rho_8$ and $\rho_9$ of degree 8 and 9 and function-germ
  $R_{10}$ with $\ord (R_{10})\ge 10$, $g_7 + \rho_8 + \rho_9 + R_{10}$ is $\R_9$-equivalent to
  $f_7 + \rho_8 + \frac{c_1}{9!}x^9 + \frac{c_2}{8!1!} x^8y$, where
  \begin{align*}
    c_1 &= \frac{\del}{\del x} \Delta^4 \rho_9 \in \RR, &
    c_2 &= \frac{\del}{\del y} \Delta^4 \rho_9 \in \RR
  \end{align*}
  \label{main-7_9}
\item For any homogeneous polynomial $\rho_8, \rho_9, \rho_{10}$ of degree 8, 9 and 10,
  $g_7 + \rho_8 + \rho_9 + \rho_{10}$ is $\R_{10}$-equivalent to $g_7 + \rho_8 + \rho_9 +
  \frac{c}{10!}x^{10}$, where
  \[
    c=\Delta^5 \rho_{10} \in \RR.
  \]
  \label{main-7_10}
\end{enumerate}
\end{thm}

To conclude this section, we introduce a related notion and we give a \corname \ of our
  classification results.

\begin{Def}[see \cite{polyharm}]
Let $l$ be a non-zero interger.
A function germ $h \colon (\RR^2, \zero) \to (\RR, \zero)$ is called an
  {\it l-harmonic function-germ} of degree $l$ if $h$ satisfies the following equation:
\[
  \Delta^l h = 0.
\]

Note that we use notation \quotes{homogeneous $l$-harmonic polynomial of degree $k$} for $l$-harmonic
  function and homogeneous function of degree $k$.
\end{Def}

\begin{cor}
Let $h_5$ be a non-zero harmonic polynomial of degree 5 and $\rho_6 \colon (\RR^2,\zero) \to
  (\RR, 0)$ be
  a homogeneous $3$-harmonic polynomial of degree $6$.
  Then $h_5 + \rho_6$ is right equivalent to $h_5$.
  \label{poly-5}
\end{cor}

\begin{proof}
Let $\rho_6$ be a homogeneous $3$-harmonic polynomial with degree $6$ i.e. $\Delta^3 \rho_6 = 0$.

By \theoname \ \ref{main-5} (\ref{main-5-change_coord}), there exists a rotation $R_{\theta} \colon
 (\RR^2, \zero) \to (\RR^2, \zero)$ and $c \in \RR$ such that
  $h_5 + \rho_6$ and $f_5 + c \rho_6 $ are right-equivalent.

Since rotation preserve Laplacian, $\Delta^3 (c\rho_6 \circ R_\theta) = c \Delta^3(\rho_6) = 0$.
Thus $h_5 + \rho_6$ is right-equivalent to $f_5$ by \theoname \ \ref{main-5} (\ref{main-5-6}).
\end{proof}

If $k=6$ or $7$, we obtain statements similar to \corname \ \ref{poly-5}:

\begin{cor}
Let $h_6$ be a non-zero harmonic polynomial of degree $6$.
\begin{enumerate}[1)]
\item Let $\rho_7$ be a homogeneous $3$-harmonic function of degree 7.
  Then $h_6 + \rho_7$ is $\R_7$ equivalent to $h_6$.
\item Let $\rho_8$ be a homogeneous $4$-harmonic function of degree 8.
  Then $h_6 + \rho_8$ is $\R_8$ equivalent to $h_6$.
\end{enumerate}
\end{cor}

\begin{cor}
Let $h_7$ be a non-zero harmonic polynomial of degree $7$.
\begin{enumerate}[1)]
\item Let $\rho_8$ be a homogeneous $3$-harmonic function of degree $8$.
  Then $h_7 + \rho_8$ is $\R_8$ equivalent to $h_7$.
\item Let $\rho_9$ be a homogeneous $4$-harmonic function of degree $9$.
  Then $h_7 + \rho_9$ is $\R_9$ equivalent to $h_7$.
\item Let $\rho_{10}$ be a homogeneous $5$-harmonic function of degree $10$.
  Then $h_7 + \rho_{10}$ is $\R_{10}$ equivalent to $h_7$.
\end{enumerate}
\end{cor}

\section{Proof of Theorems}

To prove of Theorems \ref{main-5} - \ref{main-7}, we prepare some additional notation and statements.

Let $\E_n$ denote the ring consisting of $C^\infty$ map-germs $(\RR^n, \zero) \to \RR$ and
let $m_n \subset \E_n$ denote the ideal consisting of $C^\infty$ map-germs $(\RR^n,\zero)
  \to (\RR,\zero)$.

Let $h \colon (\RR^n,\zero) \to (\RR,\zero)$ be a $C^\infty$ function-germ.
Let $J_h$ denote the Jacobian ideal of $h$ i.e. $J_h
= \langle \frac{\del h}{\del x_i}| 1 \le i \le n \rangle_{\E_n}$.

We say that $h$ is $\R_k$- determined if $h$ is right-equivalent to some germ $\tilde{h}$ with
  $j^k \tilde{h}(\zero) = j^k h(\zero)$.

We recall a sufficient condition of $\R_k$-determancy.
For the proof, see \cite{Mother-3}.

\begin{prop}[see \cite{Mother-3}]  \label{judge-determinancy} 
Let $h \colon (\RR^n,\zero) \to (\RR,\zero)$ be a $C^\infty$ map-germ.
  Suppose for a natural number $k$
\[
  m_n^k \subset m_n J_h + m_n^{k+1}
\]
holds.

Then $h$ is $\R_k$-determined.
\end{prop}

\begin{prop} 
  \label{change-harm}
Any two elements in $\Harm_k(\RR^2)\setminus\set{\zero}$ are right equivalent by a
  conformal diffeomorphism-germ.

Furthermore,
\[
  \set{\phi \colon (\RR^2,\zero) \to (\RR^2,\zero); \phi \in \GL_2(\RR),
  f_k \circ \phi = f_k} = \langle \begin{pmatrix}
    \cos \frac{2\pi}{k} & -\sin \frac{2\pi}{k} \\
    \sin \frac{2\pi}{k} & \cos \frac{2\pi}{k}
  \end{pmatrix},
  \begin{pmatrix}
    1 & 0 \\ 0 & -1
  \end{pmatrix}
  \rangle.
\]
\end{prop}

\begin{proof}
Let $h_1, h_2 \in \Harm_k(\RR^2)\setminus\set{\zero}$.
Then, by \propname \ \ref{classification-harmonic}, there exists a linear conformal
diffeomorphism-germ $\phi_1, \phi_2 \colon (\RR^2,\zero) \to (\RR^2,\zero)$ such that
$h_1 \circ \phi_1 = f_k$ and $h_2 \circ \phi_2 = f_k$.
Then, we obtain $h_2 = h_1 \circ \phi_1 \circ \phi_2^{-1}$.

Since $\phi_1 \circ \phi_2^{-1}$ is a conformal diffeomorphism-germ, first statement is proved.

Let $\phi \in \GL_2(\RR)$ and suppose that $f_k \circ \phi = f_k$.
Then $\phi(f_k^{-1}(\zero)) = f_k^{-1}(\zero)$.

Recall that $f_k = \Re (x + iy)^k$.
Then, $f_k^{-1}(\zero)$ is $k$ lines splitting plant to $2k$ parts for same angle.

There exists $c \in \RR_{\ge 0} \st$
\[
  \phi \in c\langle \begin{pmatrix}
    \cos \frac{2\pi}{k} & -\sin \frac{2\pi}{k} \\
    \sin \frac{2\pi}{k} & \cos \frac{2\pi}{k}
  \end{pmatrix},
  \begin{pmatrix}
    1 & 0 \\ 0 & -1
  \end{pmatrix}
  \rangle.
\]
However, $f_k(x, -y) = f_k \circ R_{\frac{2\pi}{k}} = f_k$.
Therefore, we obtain the statement.
\end{proof}

\begin{prop} \label{low-degree-determinancy}
Let $k$ be a natural number such that $k \le 7$.
Then $f_k$ is $\R_{\max (k, 2k-4)}$-determined.
\end{prop}

\begin{proof}
If $k=1$ then $f_1 = x$.
So by implicit theorem, for all $R_2 \in \m_2^2$, $f_1 + R_2$ is right-equivalent to $f_1$.

If $k=2$ then $f_2 = x^2 - y^2$.
So by Morse Lemma, for all $R_3 \in \m_2^3$, $f_2 + R_3$ is right-equivalent to $f_2$.

If $k=3$ then $f_3 = 3(x^2y - xy^2)$.
It is straightforward to obtain that
\[
  m_2^3 \subset m_2 J_{f_3} + m_2^4.
\]
So, $f_3$ is $\R_3$-determined by using \propname \ \ref{judge-determinancy}.

If $k=4$, $f_4 = x^4 - 6x^2y^2 + y^4$.
It is straightforward to obtain that
\[
  m_2^5 \subset m_2 J_{f_4} + m_2^6.
\]

Therefore by \propname \ \ref{judge-determinancy}, $f_4$ is $\R_5$-determined.

Furthermore for any homogeneous polynomial $p$ of degree 5, there exists a diffeomorphism germ
  $\phi \colon (\RR^2, \zero) \to (\RR^2, \zero)$ such that
  $j^5 (f_4 + p)(\zero) = j^5 (f_4 \circ \phi)(\zero)$.

Thus, we have that $f_4$ is $\R_4$-determined.

For $k=5, 6$ and $7$, we can prove that $g_k$ is $\R_{6}, \R_{8}, \R_{10}$-determined for the same
way.
\end{proof}

\begin{proof} [\proofname\  of \theoname \ \ref{main-5}]
(\ref{main-5-change_coord})
By \propname\ \ref{classification-harmonic}, there exists a linear conformal diffeomorphism-germ
  $\phi \colon (\RR^2,\zero) \to (\RR^2,\zero)$ such that $h_5 \circ \phi = f_5$.
    Then, $(h_5 + R_6)\circ \phi = f_5 + R_6 \circ \phi$ and $\ord(R_6 \circ \phi) \ge 6$.
    So, by taking $\tilde{R}_6$ as $R_6 \circ \phi$, we obtain a statement.

(\ref{main-5-6})
We define a local diffeomorphism germ $\phi \colon (\RR^2,\zero) \to (\RR^2,\zero)$ by
\begin{align*}
  \phi(x, y) &= (
    x
      + x^2(\frac{1}{5}a_6 + \frac{1}{25}a_4 + \frac{1}{25}a_2)
      + xy(\frac{1}{20}a_5 + \frac{1}{20}a_3 + \frac{1}{20}a_1)
      - y^2(\frac{a_6}{5}), \\
    & \quad y
      + x^2 (\frac{1}{80}a_5 + \frac{1}{80}a_3 + \frac{1}{16}a_1)
      - xy (\frac{7}{20}a_6 + \frac{3}{50}a_4 + \frac{1}{100}a_2) \\
    & \quad \quad - y^2 (\frac{1}{80}a_1 + \frac{1}{80}a_3 + \frac{1}{16}a_5)
).
\end{align*}.

This $\phi$ is obtained by a following way:
Let us consider the local-diffeomorphism
  \[
    \phi(x, y) = (x + s_1x^2 + s_2xy + s_3y^2, y + s_4x^2 + s_5xy + s_6y^2).
  \]

Consider equations which are obtained by setting that $x^5y, x^4y^2, x^3y^3, x^2y^4, xy^5, y^6$
  terms of $(f_5 + \rho_6)
  \circ \phi$ are equal to zero.

By solving the linear equiation of equation, we can determine $s_j(j=1, 2, 3, 4, 5, 6)$ to find
  $\phi$.

Then, by straightforward computation, we obtain
\[
  (f_5 + \rho_6) \circ \phi \cong_{j^6}
    f_5 + x^6(a_0 + a_6 + \frac{a_2 + a_4}{5})
\]

and
\[
  a_0 + a_6 + \frac{a_2 + a_4}{5} = \frac{1}{6!} \times (\Delta^3 \rho_6).
\]

(\ref{main-5-uniqueness})
By \propname \ \ref{change-harm} and the procedure to construct the
  above diffeomorphism $\phi$, all diffeomorphism $\psi$ giving the right equivalence of
  $f_5 + cx^6$ and $f_5 + \tilde{c}x^6$ has the following form:
\[
  \psi(x, y) = \phi \circ \phi_L(x, y),
\]
where, $\phi_L$ is a composition of $e^{\frac{2\pi i}{k}}$ and $\bar{z}$ and $\phi$ is a
  diffeomorphism such that
  \[
    \psi(\phi_L(f_5 + cx^6)) \cong_{j^6} f_5 + \frac{x^6}{6!} \times \Delta^3\phi_L(cx^6).
  \]

Since $\phi_L$ preserves Laplacian, we have
\[
  \Delta^3\phi_L(cx^6) = \Delta^3 cx^6 = 6! c.
\]
This proves \theoname \ \ref{main-5}.
\end{proof}

\begin{proof} [\proofname\ of \theoname\ \ref{main-6}]
(\ref{main-6_change_leading})
By \propname\ \ref{classification-harmonic}, there exists a linear conformal diffeomorphism-germ
  $\phi \colon (\RR^2,\zero) \to (\RR^2,\zero)$ such that $h_6 \circ \phi = g_6$.
    Then, $(h_6 + R_7)\circ \phi = g_6 + R_7 \circ \phi$ and $\ord(R_7 \circ \phi) \ge 7$.
    So, by taking $\tilde{R}_7$ as $R_7 \circ \phi$, we obtain a statement.

(\ref{main-6_7})
Let $\rho_7 = \sum_{j=0}^7 a_jx^{7-j}y^j$.

Define the local diffeomorphism $\phi \colon (\RR^2,\zero) \to (\RR^2,\zero)$ by

\begin{align*}
  \phi(x, y) &= (
    x
      + x^2(\frac{1}{48}a_3 + \frac{1}{24}a_5 + \frac{5}{16}a_7) \\
    & \quad\quad + xy(\frac{5}{336}a_2 +  \frac{5}{168}a_4 + \frac{19}{336}a_6)
      - y^2(\frac{a_7}{6}), \\
    &\quad y + x^2 (\frac{1}{42}a_2 + \frac{1}{70}a_4 + \frac{1}{42}a_6)
      - xy(\frac{1}{240}a_3 + \frac{1}{24}a_5 + \frac{19}{48}a_7) \\
    & \quad\quad - y^2(\frac{1}{336}a_2 + \frac{1}{168}a_4 + \frac{5}{112}a_6)).
\end{align*}

Then, by a straightforward computation
\[
  (g_6 + \rho_7) \circ \phi \cong_{j^7}
    g_6 + x^7(a_0 + \frac{3}{35}a_4 + \frac{a_2 + a_6}{7}) + x^6y(a_1 + \frac{3}{5}a_3 + a_5 + 7a_7)
    .
\]
Furthermore
\begin{align*}
  a_0 + \frac{3}{35}a_4 + \frac{a_2 + a_6}{7} &= \frac{1}{7!} \times (\frac{\del}{\del x}
    \Delta^3 \rho_7), \\
  a_1 + \frac{3}{5}a_3 + a_5 + 7a_7 &= \frac{7}{7!} \times (\frac{\del}{\del y} \Delta^3 \rho_7).
\end{align*}

(\ref{main-6_8})
Let $\rho_8 = \sum_{j=0}^8 a_j x^{8-j}y^j$.
Let us define a local diffeomorphism $\phi \colon (\RR^2,\zero) \to (\RR^2,\zero)$ by
\begin{align*}
\phi(x, y)
  &= (
    x
      - x^3(\frac{5}{768}a_1 + \frac{5}{768}a_3 + \frac{1}{256}a_5 + \frac{5}{768}a_7 ) \\
    &\quad\quad + x^2y(\frac{5}{336}a_2 + \frac{5}{168}a_4  + \frac{19}{336}a_6 + \frac{5}{12}a_8 )
      \\
    &\quad\quad + xy^2(\frac{25}{768}a_1 + \frac{5}{256}a_3 + \frac{25}{768}a_5 + \frac{47}{768}a_7)
      \\
    &\quad\quad - y^3(\frac{1}{6}a_8), \\
  & y
    + x^3(\frac{1}{42}a_2  + \frac{1}{70}a_4 + \frac{1}{42}a_6 +  \frac{1}{6}a_8) \\
  & \quad\quad+ x^2y(\frac{47}{768}a_1 + \frac{25}{768}a_3 + \frac{5}{256}a_5 +  + \frac{25}{768}a_7
    ) \\
  & \quad\quad - xy^2(\frac{5}{12}a_8 + \frac{5}{112}a_6 + \frac{1}{168}a_4 + \frac{1}{336}a_2) \\
  & \quad\quad - y^3(\frac{5}{768}a_1 + \frac{1}{256}a_3 + \frac{5}{768}a_5 + \frac{35}{768}a_7)).
\end{align*}
Then, by a straightforward computation, we obtain
\begin{align*}
(g_6 + \rho_7 + \rho_8) \circ \phi
  & \cong_\R g_6 + \rho_7 + \frac{x^8}{35} (a_0 + 5a_2 + 3a_4 + 5a_6 + 35a_8) \\
  & = g_6 + \rho_7 + \frac{x^8}{8!} \times \Delta^4 \rho_8.
\end{align*}
\theoname \ \ref{main-6} is proved.
\end{proof}

\begin{proof}[\proofname\ of \theoname\ \ref{main-7}]
(\ref{main-7_change_leading})
By \propname\ \ref{classification-harmonic}, there exists a linear conformal  diffeomorphism-germ
  $\phi \colon (\RR^2,\zero) \to (\RR^2,\zero)$ such that $h_7 \circ \phi = g_7$.
    Then, $(h_7 + R_8)\circ \phi = g_7 + R_8 \circ \phi$ and $\ord(R_8 \circ \phi) \ge 8$.
    So, by taking $\tilde{R}_8$ as $R_8 \circ \phi$, we obtain a statement.

(\ref{main-7_8})
Let $\rho_8 = \sum_{j=0}^8 a_jx^{8-j}y^j$.

Let the local diffeomorphism $\phi \colon (\RR^2,\zero) \to (\RR^2,\zero)$ be defined by
\begin{align*}
\phi(x, y)
&:=
  (x
    + x^2(\frac{5}{448}a_3 + \frac{3}{224}a_5 + \frac{15}{448}a_7) \\
    & - xy(\frac{1}{490}a_4 + \frac{3}{98}a_6 + \frac{3}{7}a_8)
    - y^2(\frac{3}{3136}a_3 + \frac{5}{1568}a_5 + \frac{15}{448}a_7), \\
  & y
    - x^2(\frac{3}{245}a_4 + \frac{2}{49}a_6 + \frac{3}{7}a_8) \\
    & - xy(\frac{9}{1568}a_3 + \frac{15}{784}a_5 + \frac{13}{224}a_7)
    + y^2(\frac{a_8}{7})).
\end{align*}

Then, by a direct computation,
\begin{align*}
  (g_7 + \rho_8)\circ \phi &\cong_{j^8}
      g_6 + \frac{x^7}{7!} \times (\frac{\del}{\del x}\Delta^3 \rho_7)
      + \frac{7x^6y}{7!} \times (\frac{\del}{\del y}\Delta^3 \rho_7) \\
      & \quad + \frac{28x^6y^2}{8!} \times (\frac{\del^2}{\del y^2}\Delta^3 \rho_8).
\end{align*}

(\ref{main-7_9})
Let $\rho_9 = \sum_{j=0}^9 a_jx^{9-j}y^j$.

Let local diffeomorphism $\phi \colon (\RR^2,\zero) \to (\RR^2,\zero)$ be defined by
\begin{align*}
\phi(x, y)
&:=
  (x + x^3(\frac{5}{448}a_3 + \frac{3}{224}a_5 + \frac{15}{448}a_7 + \frac{5}{16}a_9) \\
    & \quad + x^2y(\frac{23}{1344}a_2 + \frac{29}{1568}a_4 + \frac{65}{3136}a_6 + \frac{17}{336}a_8)
      \\
    & \quad - xy^2(\frac{3}{3136}a_3 + \frac{5}{1568}a_5 + \frac{105}{3136}a_7 + \frac{51}{112}a_9)
      \\
    & \quad - y^3(\frac{5}{4032}a_2 + \frac{1}{672}a_4 + \frac{5}{1344}a_6 + \frac{5}{144}a_8), \\
  & y + x^3(\frac{1}{63}a_2 + \frac{1}{147}a_4 + \frac{1}{147}a_6 + \frac{1}{63}a_8) \\
    & \quad - x^2y(\frac{9}{1568}a_3 + \frac{15}{784}a_5 + \frac{13}{224}a_7 + \frac{33}{56}a_9) \\
    & \quad - xy^2(\frac{5}{672}a_2 + \frac{1}{112}a_4 + \frac{5}{224}a_6 + \frac{11}{168}a_8)
    + y^3(\frac{a_9}{7})
).
\end{align*}

Then, by a straightforward computation,
\begin{align*}
(g_7 + \rho_8 + \rho_9) \circ \phi \cong_{j^9}
    g_7 + \rho_8
    + \frac{x^9}{9!} \times (\frac{\del}{\del x}\Delta^4 \rho_9)
    + \frac{9 x^8y}{9!} \times (\frac{\del}{\del y} \Delta^4 \rho_9)
\end{align*}

(\ref{main-7_10})
Let $\rho_{10} = \sum_{j=0}^{10} a_jx^{10-j}y^j$.

Let local diffeomorphism $\phi \colon (\RR^2,\zero) \to (\RR^2,\zero)$ be defined by

\begin{align*}
\phi(x,y) &=
  (x
     - x^4(\frac{9}{256}a_1 + \frac{1}{256}a_3 + \frac{3}{1792}a_5
     + \frac{3}{1792}a_7) + \frac{1}{256}a_9 \\
  &+ x^3y(\frac{23}{1344}a_2 + \frac{29}{1568}a_4 + \frac{65}{3136}a_6 + \frac{17}{336}a_8
    + \frac{209}{448}a_{10}) \\
  & + x^2y^2(\frac{51}{896}a_1 + \frac{3}{128}a_3 + \frac{19}{128}a_5 +
    \frac{3}{128}a_7 + \frac{51}{896}a_9 ) \\
  & -xy^3(\frac{5}{4032}a_2 + \frac{1}{672}a_4 + \frac{5}{4032}a_6
    +\frac{5}{144}a_8  + \frac{209}{448}a_{10}) \\
  & -y^4(\frac{1}{256}a_1 + \frac{3}{1792}a_3 + \frac{3}{1792}a_5 + \frac{1}{256}a_7
     + \frac{9}{256}a_9), \\
  & y
    + x^4(\frac{1}{63}a_2 + \frac{1}{147}a_4 + \frac{1}{147}a_6
      \frac{1}{63}a_8 + \frac{1}{7}a_{10}) \\
  & + x^3y(\frac{61}{896}a_1 + \frac{3}{128}a_3
    + \frac{9}{896}a_5 + \frac{9}{896}a_7 + \frac{3}{128}a_9) \\
  & - x^2y^2(\frac{5}{672}a_2 + \frac{1}{112}a_4 +\frac{5}{224}a_6 + \frac{11}{168}a_8
      + \frac{21}{32}a_{10}) \\
  & - xy^3(\frac{3}{128}a_1 + \frac{9}{896}a_3 + \frac{9}{896}a_5 + \frac{3}{128}a_7
    + \frac{61}{896}a_9 ) \\
  & + y^4(\frac{a_{10}}{7})).
\end{align*}

Then, by a direct computation,
\begin{align*}
(g_7 + \rho_8 + \rho_9 + \rho_{10}) \circ \phi \cong_{j^{10}}
    g_7 + \rho_8 + \rho_9 + \frac{x^{10}}{10!} \times (\Delta^5 \rho_{10}).
\end{align*}
Thus, we have \theoname \ \ref{main-7}.
\end{proof}



\end{document}